\documentclass[11pt,a4paper,reqno]{amsart}
\usepackage[utf8]{inputenc}
\usepackage[T1]{fontenc}
\usepackage[english]{babel}
\usepackage[top=3.7cm,bottom=3.3cm,left=2.5cm,right=2.5cm,heightrounded,bindingoffset=0mm]{geometry}
\usepackage{amsmath,amssymb,amsthm,amsfonts}
\usepackage{xcolor}
\usepackage{esint}
\usepackage{tikz}
\usepackage{hyphenat}
\usepackage{enumitem}
\usepackage{thmtools}

\theoremstyle{plain}
\newtheorem{theorem}{Theorem}[section]
\newtheorem{lemma}[theorem]{Lemma}
\newtheorem{proposition}[theorem]{Proposition}
\newtheorem{question}{Question}
\newtheorem{conj}{Conjecture}

\theoremstyle{remark}
\newtheorem{remark}[theorem]{Remark}

\newcommand{\dd}{\mathop{}\!\mathrm{d}}

\usepackage{microtype}
\usepackage{graphicx}
\usepackage{csquotes}

\author{Pierre Germain}
\address{
Fakult\"at f\"ur Mathematik\\
Oskar-Morgenstern-Platz 1\\ 
1090 Wien, Austria
}
\email[P. Germain]{pierre.germain@univie.ac.at}

\author{Micka\"el Latocca}
\address{
Laboratoire de Mathématiques et de Modélisation d'\'Evry (LaMME)\\
Université d'\'Evry Paris-Saclay\\
23 Bd François Mitterrand, 91000 Évry-Courcouronnes, France
}
\email[M. Latocca]{mickael.latocca@univ-evry.fr}

\title[Periodic Liouville Strichartz estimates]{Strichartz estimates for the Liouville equation on Euclidean tori and applications to Kakeya}



\makeatletter
\renewcommand{\@@and}{\&}
\makeatother

\usepackage[colorlinks=true,hidelinks]{hyperref}
\usepackage[nameinlink]{cleveref}
\crefname{conj}{Conjecture}{Conjectures}
\crefname{proposition}{Proposition}{Propositions}
\crefname{lemma}{Lemma}{Lemmas}
\crefname{question}{Question}{Questions}

\date{\today}
\thanks{}

\begin{document}

\begin{abstract}
We prove Strichartz estimates for the space-time density $\rho$ of solutions to the free Liouville equation on flat tori. In dimension one, we obtain the optimal range of estimates for the density $\rho \in L^p_{t,x}$ in terms of $f_0 \in L^{a}_vL^{b}_x$. In higher dimensions, we prove that such estimates cannot hold and that a weight has to be added: $\rho$ can be bounded in terms of the norm of $|v|^\gamma f_0$. We conjecture a range of optimal estimates, and partially prove them. Finally, these results have natural applications to the $X$-ray transform and Kakeya problems on Euclidean cylinders.
\end{abstract}

\hrule
\vspace{0.5cm}
\maketitle
\hrule 
\vspace{0.4cm}

\section{Introduction}

\subsection{The Liouville equation on the torus and its asymptotic behavior}

This article is dedicated to estimates for the free Liouville equation (or kinetic transport equation)
\begin{equation}\label{eq:Liouville}
    \partial_t f(t,x,v) + v\cdot\nabla_x f(t,x,v) = 0,\qquad f(0,x,v)=f_0(x,v) \geq 0, 
\end{equation}
set on the torus: $(x,v) \in \mathbb{T}^d \times \mathbb{R}^d$ with
\[
    \mathbb{T}^d = \mathbb{R}^d / (2\pi \mathbb{Z}^d).
\]
The solution is given by the explicit formula
\[
    f(t,x,v)=f_0(x-tv,v), 
\]
with corresponding density 
\[
    \rho(t,x)=\int_{\mathbb{R}^d} f_0(x-tv,v)\dd v.
\]

Space-independent solutions $f(t,x,v) = g(v)$ are stationary, and through mixing they are expected to play the role of attractors in appropriate topologies. At the level of the density, this means that 
\[
    \rho(t,x) \overset{t\to \infty}{\longrightarrow} \frac{1}{(2\pi)^d}\int_{\mathbb{R}^d} \int_{\mathbb{T}^d} f_0(x,v) \dd x \dd v.
\]

In this paper, we aim at quantifying this phenomenon. Upon subtracting from $f_0$ its space average, we consider from now on densities $f_0$ that have zero average in $x$ (and in particular we will jettison the non-negativity hypothesis), and we ask the following question.

\begin{question}\label{question1}
For which indices $p,a,b$ does there hold
\begin{equation}
\label{withoutweight}
    \| \rho(t,x) \|_{{L^p_{t,x}}(\mathbb{R} \times \mathbb{T}^d)} \lesssim \|  f_0(x,v) \|_{L^a_v(\mathbb{R}^d,L^b_x(\mathbb{T}^d))} \qquad \mbox{if} \;\;\; \int f_0(x,v) \dd x = 0 \quad \mbox{for all $v \in \mathbb{R}^d$}
\end{equation}
or more generally for which $p,\gamma,a,b$
\begin{equation}
\label{withweight}
    \| \rho(t,x) \|_{{L^p_{t,x}}(\mathbb{R} \times \mathbb{T}^d)} \lesssim \|  |v|^\gamma f_0(x,v) \|_{L^a_v(\mathbb{R}^d,L^b_x(\mathbb{T}^d))} \qquad \mbox{if} \;\;\; \int f_0(x,v) \dd x = 0 \quad \mbox{for all $v \in \mathbb{R}^d$} \;\; ?
\end{equation}
\end{question}

Why such a choice of norms? On the left-hand side, this is the simplest choice. On the right-hand side, this set of norms is natural since it is invariant by the flow of the Liouville equation (in a way which is reminiscent of the dispersive Strichartz estimates, to which we will come back). Furthermore, the weight $\gamma>0$ is needed in dimension $\geq 2$ to obtain non-trivial estimates. Finally, a more general set of invariant norms is $\||v|^{\gamma}|D_x|^{\delta} \cdot \|_{L^a_vL^b_x}$, but we set $\delta=0$ for simplicity.

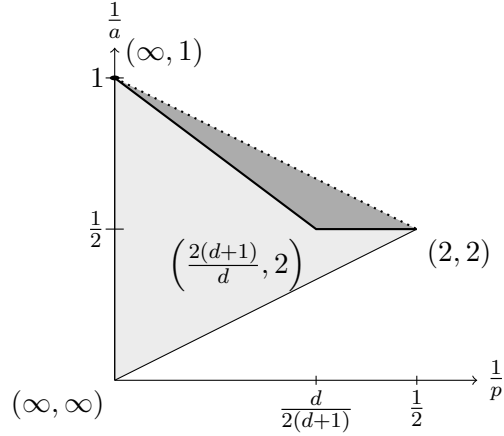
\begin{figure}[ht]
\centering
\begin{tikzpicture}[x=8cm,y=4cm]
\coordinate (Pinfinf) at (0,0);
\coordinate (Pinf1)   at (0,1);
\coordinate (O)       at ({1/2},{1/2});
\coordinate (P32)     at ({1/3},{1/2});
\fill[gray!15] (Pinfinf) -- (Pinf1) -- (O) -- cycle;
\fill[gray!65] (Pinf1) -- (P32) -- (O) -- cycle;
\draw[->] (0,0) -- (0.60,0) node[right] {$\frac{1}{p}$};
\draw[->] (0,0) -- (0,1.10) node[above] {$\frac{1}{a}$};
\def\tick{0.015}
\draw ({1/3},-\tick) -- ({1/3},\tick);
\draw ({1/2},-\tick) -- ({1/2},\tick);
\draw (-\tick,{1/2}) -- (\tick,{1/2});
\draw (-\tick,1) -- (\tick,1);
\node[below] at ({1/3},0) {$\frac{d}{2(d+1)}$};
\node[below] at ({1/2},0) {$\frac{1}{2}$};
\node[left] at (0,{1/2}) {$\frac{1}{2}$};
\node[left] at (0,1) {$1$};\draw[thin] (Pinfinf) -- (Pinf1);
\draw[thin] (Pinfinf) -- (O);
\draw[dotted,thick] (Pinf1) -- (O);
\draw[solid,thick] (Pinf1) -- (P32) -- (O);
\fill (Pinf1) circle (0.008);
\node[below left] at (Pinfinf){$(\infty,\infty)$};
\node[above right] at (Pinf1){$(\infty,1)$};
\node[below right] at (O) {$(2,2)$};
\node[below left] at (P32) {$\left(\frac{2(d+1)}{d},2\right)$};
\end{tikzpicture}
\caption{We consider the inequality \eqref{eq.homogenous-d} with $\gamma= \frac{d}{a'}-\frac{1}{p}$ and $b = \frac{dp}{d+1}$, and ask for which values of $(a,p)$ it holds true.
The conjectured region is in gray, and the region where we could prove the inequality is dark gray.}
\label{figue}
\end{figure}

\subsection{Main results} In dimension $1$, it is possible to omit the weight $\gamma$ and obtain the following result.

\begin{theorem}[Euclidean tori: dimension $d=1$]\label{thm.torus.1d}
For all $f_0=f_0(x,v)$ such that $\int_{\mathbb{T}} f_0(x,v)\dd x=0$, there holds 
\begin{equation}\label{conjmean0}
    \| \rho \|_{L^p_{t,x}(\mathbb{R}\times \mathbb{T})} \lesssim \| f_0 \|_{L^{a}_v L^{b}_x}, \qquad a = p', \qquad b\geq \frac{p}{2}, \qquad p>2.
\end{equation}
Moreover, these are the only values of $(p,a,b) \in [1,\infty]^3$ for which such an estimate holds.
\end{theorem}

In dimension larger or equal to $2$, the weight $\gamma$ is needed, and our result is as follows.

\begin{theorem}[Euclidean tori: dimension $d\geq 2$]\label{thm.noest-dim2}
We consider the inequality
\begin{equation}
    \label{eq.homogenous-d}
    \|\rho\|_{L^p(\mathbb{R}_t \times \mathbb{T}^d)} \lesssim \| |v|^{\gamma} f_0\|_{L^a_vL^b_x}, \quad \int_{\mathbb{T}^d}f_0(x,v)\mathrm{d}x = 0, 
\end{equation}
where $(p,a,b,\gamma) \in [1,\infty]^3 \times \mathbb{R}$. 
\begin{itemize}
\item[(i)] Necessary condition: the inequality \eqref{eq.homogenous-d} can only hold if
\begin{equation}
\label{necessary}
    \gamma = \frac{d}{a'}-\frac{1}{p}, \qquad p\geq 2, \qquad  p'<a\leq p, \qquad b\geq \frac{dp}{d+1}.
\end{equation}
\item[(ii)] Sufficient condition including the endpoint in $b$: the inequality \eqref{eq.homogenous-d} is satisfied if \eqref{necessary} holds and furthermore
\[ 
    2<p<\infty, \qquad p'<a < \min\left\{2,\frac{dp}{dp-(d+1)}\right\}.
\]
\item[(iii)] Sufficient condition excluding the endpoint in $b$: the inequality \eqref{eq.homogenous-d} is satisfied if \eqref{necessary} holds and
\[
    p'<a < p, \qquad b > \frac{dp}{d+1}.
\]
\end{itemize}
\end{theorem}

\begin{remark}
Including the endpoint in $b$ in the above estimates is particularly relevant, since this corresponds to the scale-invariant case.
\end{remark}

Both theorems above imply Strichartz estimates for general data: if
\[
    f_0(x,v) = g_0(v) + h_0(x,v), \qquad \mbox{with} \;\;\; \int h_0(x,v) \dd x = 0,
\]
then
\begin{equation*}
    \|\rho \|_{L^p_{t,x}([0,1]\times \mathbb{T}^d)} \lesssim \|g_0\|_{L^1_v} + \||v|^\gamma h_0 \|_{L^{a}_v L^{b}_x},
\end{equation*}
provided that $p,\gamma,a,b$ satisfy the hypotheses of \Cref{thm.torus.1d} or \Cref{thm.noest-dim2}. Notice however that the time interval in the above inequality is now finite.

\subsection{Applications}

\subsubsection{The $X$-ray transform on the cylinder} The dual of the map $f_0(x,v) \mapsto \rho(t,x)$ is the modified $X$-ray transform
\[
    [\widetilde{X} f](x,v) = \int_{-\infty}^{+ \infty} f(t,x+tv) \dd t
\]
(it differs from the standard $X$-ray transform by a factor $\sqrt{1+|v|^2}$).

Given a function in $(x,v) \in \mathbb{T}^d \times \mathbb{R}^d$, we denote $\Pi_0$ and $\Pi_{\neq}$ for the orthogonal projections on functions which do not depend on $x$, and functions with zero mean in $x$ respectively:
\[
    [\Pi_0 f](x,v) = \frac 1 {(2\pi)^d} \int_{\mathbb{T}^d} f(x,v) \dd x, \qquad \Pi_{\neq} = \operatorname{Id} - \Pi_0.
\]

By duality, the boundedness of
\[
    \rho \Pi_{\neq} |v|^{-\gamma}:\, L_v^a (\mathbb{R}^d, L^b_x(\mathbb{T}^d)) \, \longrightarrow \, L^p_{t,x}(\mathbb{R} \times \mathbb{T}^{d}),
\]
which is addressed in \Cref{thm.torus.1d,thm.noest-dim2}, is equivalent to that of
\[
    |v|^{-\gamma} \Pi_{\neq} \widetilde{X} : \, L^{p'}_{t,x}(\mathbb{R} \times \mathbb{T}^{d}) \, \longrightarrow \, L_v^{a'} (\mathbb{R}^d, L^{b'}_x(\mathbb{T}^d)).
\]
\subsubsection{The Kakeya problem on the cylinder} 
Consider a family of tubes $\{ T_k \}$ of length $R$ and cross section of radius $\delta$, and $\frac{\delta}{R}$-separated directions---this might be on Euclidean space $\mathbb{R}^{d+1}$ or on the cylinder $\mathbb{R} \times \mathbb{T}^d$. 
In the whole space, the parameter $R$ is usually normalized to $1$, which is always possible by scaling. But the Kakeya problem on the cylinder is only relevant if $R \to \infty$, both cases being equivalent for bounded $R$.

The Kakeya maximal operator conjecture on Euclidean space states that, for any $\varepsilon>0$,
\begin{equation}
\label{KMO}
    \left\| \sum_k \mathbf{1}_{T_k} \right\|_{L^p(\mathbb{R}^{d+1})} \lesssim 
    \begin{cases}
        \left( \frac{R}{\delta} \right)^\varepsilon R^{\frac{d+1}p} & \mbox{if $p \leq \frac{d+1}{d}$} \\
        \left( \frac{R}{\delta} \right)^\varepsilon R^d \delta^{-d + \frac{d+1}{p}} & \mbox{if $p > \frac{d+1}{d}$} .
    \end{cases}
\end{equation}
We refer to \cite{Cordoba} for the classical proof of the case $d=2$, and to \cite{Mattila} for a textbook presentation.

We now turn to the case of the cylinder. For simplicity, we focus on the case of tubes with direction $(1,v) \in \mathbb{R}^{d+1}$ with $|v| \lesssim 1$; the general case would only involve an additional weight, which is natural since the combinatorial properties of tubes depend critically on the size of $v$.

\begin{theorem}[Maximal operator Kakeya on the cylinder] \label{thm.Kakeya}
On the cylinder $\mathbb{R}\times \mathbb{T}^d$, consider a family of tubes $\{T_k\}_{k\in \mathcal{K}}$ with length $R$, cross section of radius $\delta$, and directions $(1,v_k)$, where the $\{v_k\}$ are $\frac{\delta}{R}$-separated and of size $\sim 1$. Then for all $p>2$ there holds
\begin{equation}
\label{sauterelle}
    \left\| \sum_{k\in\mathcal{K}}\mathbf{1}_{T_k} \right\|_{L^p(\mathbb{R}\times \mathbb{T}^{d})} \lesssim R^{d+\frac{1}{p}} + R^d \delta^{\frac{d+1}{p}-d}.
\end{equation}
Furthermore, this inequality is optimal: it can be saturated.
\end{theorem}

The proof of this theorem will be given in \Cref{section_proof_Kakeya}; it relies on the boundedness of $\widetilde{X}$ which can be deduced by duality from \Cref{thm.torus.1d,thm.noest-dim2}. Without entering into details for now, it is interesting to give a geometric interpretation of the two terms on the right-hand side of \eqref{sauterelle}:
\begin{itemize}
    \item The term $R^d \delta^{\frac{d+1}{p}-d}$ corresponds to the  ``bush'' example, where all tubes intersect at one point; this example is the one which saturates the inequality \eqref{KMO} in the whole space for $p \geq \frac{d+1}{d-1}$;
    \item The term $R^{d+\frac{1}{p}}$ can be understood as an average interaction of tubes which fill up the torus---this effect is absent in the case of the whole space.
\end{itemize}

\subsection{Related directions} 

\subsubsection{Strichartz estimates for the Schr\"odinger equation} 

This is where it all started! After Strichartz' original paper \cite{Strichartz} regarding estimates in the whole space, Bourgain \cite{Bourgain} initiated the study of estimates on the torus of the type
\[
    \| e^{it\Delta} f_0 \|_{L^p_{t,x}([0,1] \times \mathbb{T}^d)} \lesssim \| f_0 \|_{H^s(\mathbb{T}^d)},
\]
(where $p \in [1,\infty], s\geq 0$) which was completed in \cite{BourgainDemeter}. One can also ask for estimates of the type
\[
    \| e^{it\Delta} f_0 \|_{L^p_{t,x}([0,1] \times \mathbb{T}^d)} \lesssim \| \widehat{f_0} \|_{L^q(\mathbb{T}^d)}
\]
(where $p,q \in [1,\infty]$) but finding the optimal range of $(p,q)$ is an open problem, equivalent to the Fourier restriction problem.

The estimates above on $e^{it\Delta} f_0$ and \eqref{withoutweight}, \eqref{withweight} have common features: the left-hand side of the equation is a space-time $L^p$ norm, while the right-hand side norm is invariant through the flow of the linear problem (Liouville of Schr\"odinger). Furthermore, the Schr\"odinger and Liouville equations on the torus are intimately tied, since they are exchanged by the Wigner transform---though it does not seem possible to use this correspondence to derive analytical estimates.

\subsubsection{Euclidean space} 

The analog of \Cref{question1} can be asked if the space variable ranges over $\mathbb{R}^d$; this is a classical subject which is the background of the present article. It will be surveyed in \Cref{section_Euclidean_space}, suffice it for now to say that the conjecture
\[
    \| \rho \|_{L^p_{t,x}(\mathbb{R} \times \mathbb{R}^d)} \lesssim \| f_0 \|_{L^a_v (\mathbb{R}^d, L^b_x(\mathbb{R}^d))} \quad \mbox{if} \quad  
a'=dp, \;\; b = \frac {dp}{d+1}, \;\; p > \frac{d+1}{d}
\]
is known if $d=1$ or $p \geq \frac{d+2}{d}$.

\subsubsection{General compact manifolds} 

The Liouville equation can be generalized in a natural way to compact (or geometrically bounded) Riemannian manifolds $M$, by replacing transport along straight lines by transport along geodesics. In this context, Salort \cite{Salort2007} proved the following inequality
\begin{align*}
    & \|f\|_{L_t^q([0,1],L^p_x(M,L^r_v(T^*_xM)))} \lesssim \|\langle v\rangle^{\frac{1}{q} + \varepsilon}f_0\|_{L^a_{x,v}(TM)} \\
    & \qquad \qquad \mbox{if} \qquad  \frac{2}{q} = d\left(\frac{1}{r} - \frac{1}{p}\right), \quad \frac{1}{a} = \frac{1}{2}\left(\frac{1}{r} + \frac{1}{p}\right), \quad q>2\geq a.
\end{align*}
Furthermore, the value of $q$ is optimal, as can be seen from considering stationary solutions of the form $g(|v|)$.

The strength of this estimate is its universality, but its weakness is that it is limited to a finite time interval. As we saw in the case of the torus, global in time estimates require additional assumptions and quantify relaxation to the equilibrium through mixing. Proving such estimates for general manifolds seems to be very difficult, but it is also very interesting!

\subsubsection{Velocity averaging lemmas} 

Velocity averaging lemmas play a crucial role in the theory of kinetic equations \cite{GLPS,GolseStRaymond,Perthame2002}. They are usually formulated in the inhomogeneous case, but as observed in \cite[Remark 3.2]{Perthame2004}, homogeneous versions follow. The basic instance is
\[
    \| \chi(t) \rho_\psi(t,x) \|_{H^{1/2}_{t,x}} \lesssim \| f_0 \|_{L^2_{x,v}}, \qquad \rho_\psi(t,x) = \int_{\mathbb{R}^d} \psi(v) f_0(x-tv,v) \dd v,
\]
where $f_0(x-tv,v)$ is the solution of the Liouville equation set in the whole space, and $\chi,\psi \in \mathcal{C}^\infty_c$. This estimate gives improved regularity for $\rho_\psi$ compared to $f_0$, which can be converted by Sobolev embedding into improved integrability. This is also a feature of our Strichartz estimates, which do give improved integrability---but the cutoff functions appearing above make this estimate local in time and velocity, which is quite different from what we are aiming at.

\subsubsection{Regularity in $v$ and Landau damping} After decomposing $f$ in Fourier series in $x$
\[
    f(x,v) = \sum_{k \in \mathbb{Z}^d} f_k(v) e^{ik\cdot x},
\]
it is easily computed that
\[
    \rho(t,x) = (2\pi)^{d/2} \sum_k \widehat{f_k}(kt) e^{ik\cdot x},
\]
where $\widehat{f_k}$ stands for the Fourier transform in $v$ of $f_k$. It is apparent from this formula that regularity in $v$ translates into time-decay for $\rho$. See \cite{Bedrossian}, Section 2 for a nice discussion and the connection with recent developments around Landau damping \cite{BedrossianMouhotVillani,MouhotVillani}.

Though greater regularity in $v$ of the initial data immediately gives better decay in $t$, our aim in the present paper is to obtain decay without assuming any regularity in $v$!

\subsection{Notations} If $f$ is a function on $\mathbb{R}^d$, we adopt the following normalization for the Fourier transform:
\[
    \widehat{f}(\xi) = \frac{1}{(2\pi)^{d/2}} \int_{\mathbb{R}^d} f(x) e^{-i x \cdot \xi} \dd x, \qquad f(x) = \frac{1}{(2\pi)^{d/2}} \int_{\mathbb{R}^d} \widehat{f}(\xi) e^{i x \cdot \xi} \dd \xi.
\]

If $A$ and $B$ are two quantities, we write
\begin{itemize}
    \item $A \lesssim B$ if there exists a constant $C$ such that $A \leq CB$,
    \item $A \sim B$ if $A \lesssim B$ and $B \lesssim A$,
    \item $A \ll B$ if $A \leq CB$ for a sufficiently small (depending on the context) constant $C$.
\end{itemize}

\subsection*{Acknowledgements} P.G. was supported by a startup grant from the Universit\"at Wien.

\section{The case of Euclidean space}
\label{section_Euclidean_space}

This section focuses on the Strichartz problem for the Liouville equation set in the whole Euclidean space. We review the state of the art and give short proofs of the known results, with some new ideas.

\subsection{The conjectures and their equivalent formulations} On $\mathbb{R}^d$, the density is given in terms of the density through the same formula as on $\mathbb{T}^d$, namely
\[
    \rho(t,x) = \int_{\mathbb{R}^d} f_0(x-tv,v) \dd v.
\]
We focus here on the question of boundedness of $f_0 \to \rho$ from $L^a_v L^b_x$ to $L^p_{t,x}$. The following conjecture appeared in \cite{Christ}, though in the dual form which will be explained below.

\begin{conj}
\label{conjecture1}
There holds
\[
    \| \rho \|_{L^p_{t,x}(\mathbb{R} \times \mathbb{R}^d)} \lesssim \| f_0 \|_{L^a_v (\mathbb{R}^d, L^b_x(\mathbb{R}^d))}
\]
if 
\[
    a'=dp, \;\;\;\;\;\; b = \frac {dp}{d+1}, \;\;\;\;\;\; p > \frac{d+1}{d}.
\]
\end{conj} 

Notice that, as opposed to the case of the torus, a weight in $v$ is not needed.

\medskip
\noindent \underline{The dual problem:} The dual of $f_0 \mapsto \rho$ is the modified $X$-ray transform
\[
    [\widetilde{X}f](x,v) = \int_{-\infty}^{+\infty} f(t,x+tv,v) \dd t,
\]
which maps functions on $\mathbb{R}\times\mathbb{R}^{d}$ to functions on $\mathbb{R}^{d}\times\mathbb{R}^d$.

\medskip

\noindent \underline{The $X$-ray transform.} Given a line $\ell$ in $\mathbb{R}^{d+1}$, we denote $\dd \sigma_\ell$ the associated surface measure and define the $X$-ray transform on functions on $\mathbb{R}^{d+1}$ through
\[
    [Xf](\ell) =  \int_\ell f \dd \sigma_\ell.
\]
Lines in $\mathbb{R}^{d+1}$ are naturally parameterized by their direction $\theta \in \mathbb{S}^d$ and the intersection $y$ of the line with the hyperplane orthogonal to $\theta$, which can be identified with $\mathbb{R}^d$. 

If $(x,v) \in \mathbb{R}^d \times \mathbb{R}^d$, we denote $\ell_{x,v}$ for the line in $\mathbb{R}^{d+1}$ through $(0,x)$ with direction $(1,v)$. Then
\[
    [X f](\ell_{x,v}) = \sqrt{1+v^2}\, [\widetilde{X} f](x,v).
\]
Observe that $X$ and $\widetilde{X}$ only differ by a bounded factor if $v$ is bounded.

A natural set of mixed norms in the set of lines is $L^{q_1}_\theta (\mathbb{S}^d,L^{q_2}_y(\mathbb{R}^{d}))$, and the relevant question the boundedness of $X$ from $L^r(\mathbb{R}^{d+1})$ to $L^{q_1}_\theta (\mathbb{S}^d,L^{q_2}_y(\mathbb{R}^{d}))$.

\medskip
\noindent \underline{Equivalent formulations of the conjecture.}
Under the scaling assumption 
\[
    a' =dp, \qquad b = \frac{dp}{d+1},
\]
the following statements are equivalent: 
\begin{equation*}
\begin{array}{rrcl}
    \rho : & L^a_v(\mathbb{R}^d, L^b_x(\mathbb{R}^d)) &\longrightarrow &L^p_{t,x}(\mathbb{R}^{d+1}) \\
    \widetilde{X} : &L^{p'}_{t,x}(\mathbb{R}^{d+1}) &\longrightarrow & L^{a'}_v(\mathbb{R}^d, L^{b'}_x(\mathbb{R}^d)) ,\\
    X : &L^{p'}_{t,x}(\mathbb{R}^{d+1}) &\longrightarrow &L^{a'}_\theta(\mathbb{S}^d, L^{b'}_y (\mathbb{R}^{d+1}).
\end{array}
\end{equation*}
Indeed, the two first assertions are equivalent by duality. To see that the second implies the third, we restrict velocities to $|v| \leq 1$ and observe that $X$ and $\widetilde{X}$ are equivalent, and the mixed norms as well. This gives boundedness of $X$ for lines whose directions are restricted to a cap of $\mathbb{S}^d$, and the full boundedness follows by isotropy. Finally, the same argument gives that the third assertion implies the boundedness of $\widetilde{X}$ restricted to velocities $\leq 1$, which can then be extended to all velocities by scaling.

\subsection{Known results}

The history of \Cref{conjecture1} is as follows
\begin{itemize}
    \item In case $d=1$, the $X$-ray transform coincides with the Radon transform, and the result is due to Oberlin--Stein \cite{OberlinStein}.
    \item The case $p \geq \frac{d+2}{d}$ was proven by Christ \cite{Christ} improving on Drury \cite{Drury} who was missing the endpoint.
    \item The case $p = \frac{d+2}{d}$ was rediscovered by Keel--Tao \cite{keelTao1998} who noticed that the result in Castella--Perthame \cite{CastellaPerthame1996} could be extended.
    \item Finally, the case $p=\frac{d+2}{d}$ can be viewed as a consequence of the Strichartz estimates for orthonormal families \cite{Sabin}.
\end{itemize}

Finally, another line of research \cite{LabaTao,Wolff,Oberlin} focuses on the case where $x$ and $v$ are bounded. This setting is directly related to the Kakeya problem and it allows to relax some of the constraints on $a$, $b$ and $p$.

\subsection{Proof of the case \texorpdfstring{$d=1$}{d=1} of \Cref{conjecture1}}

This result is equivalent to the optimal $X$-ray transform estimate in dimension 2, which was originally proved in Oberlin--Stein \cite{OberlinStein}. We propose here a different proof, inspired by the treatment of the Fourier restriction problem in dimension 2 by H\"ormander \cite{Hormander}.

\medskip

We  consider first the operator
\begin{equation}
    \label{eq.defT}
    \mathcal{T}[F](t,x) = \iint F(x-tv,x-tw,v,w) \dd v \dd w,
\end{equation}
which is linear in $F = F(y,z,v,w)$. On the one hand, we can see that
\[
    \| \mathcal{T}[F] \|_{L^\infty_{t,x}} \lesssim \| F \|_{L^1_{v,w} L^\infty_{y,z}}.
\]
On the other hand, using that the change of variables $(t,x) \mapsto (x-tv,x-tw)$ has Jacobian $|v-w|$, we get the estimate
\[
\| 
    \mathcal{T}[F] \|_{L^1_{t,x}} \leq \iiiint |F(x-tv,x-tw,v,w)| \dd v \dd w \dd t \dd x = \| F \|_{L^1_{v,w}(|v-w|^{-1} \dd v \dd w) L^1_{y,z}}.
\]
We interpolate between these two estimates and obtain 
\begin{equation}\label{TinLp}
    \|\mathcal{T}[F] \|_{L^p_{t,x}} \lesssim \| F \|_{L^1_{v,w}(|v-w|^{-\frac{ 1}{p}} \dd v \dd w) L^p_{y,z}} \qquad \text{if }1 \leq p \leq \infty.
\end{equation}
More specifically, we use \cite[Theorem 5.4.1]{BerghLofstrom} applied to the measures $\mathrm{d}v\mathrm{d}w$  and $|v-w|^{-1}\mathrm{d}v\mathrm{d}w$.
We now write
\[
    \rho^2(t,x) = \iint f(x-tv,v) f(x-tw,w) \dd v \dd w = \iint F(x-tv,x-tw,v,w) \dd v \dd w
\]
after setting
\[
    F(y,z,v,w) = f(y,v) f(z,w).
\]
By the estimate \eqref{TinLp}, we can bound for $p>2$
\begin{align*}
    \| \rho \|_{L^p_{t,x}}^2 & = \| \mathcal{T}[F] \|_{L^{\frac p2}_{t,x}} \lesssim \| F \|_{L^1_{v,w}(|v-w|^{-\frac 2p} \dd v \dd w) L^{\frac p2}_{y,z}} \\
    & = \iint |v-w|^{-\frac 2 p} \| f(y,v) \|_{L^{\frac p2}_y} \| f(z,w) \|_{L^{\frac p2}_z} \dd v \dd w \\
    & \lesssim \| f \|_{L^{p'}_v L^{\frac p2}_x}^2,
\end{align*}
where we have used the Hardy--Littlewood--Sobolev inequality in the last line. 

\subsection{The proof of the case \texorpdfstring{$p \geq \frac{d+2}{d}$}{p >= (d+2)/d} of \Cref{conjecture1}} By interpolation, it suffices to prove the case $p=\frac{d+2}{d}$, which is the inequality
\[
    \| \rho \|_{L^{\frac{d+2}{d}}_{t,x}} \lesssim \| f_0 \|_{L^{\frac{d+2}{d+1}}_{x,v}}.
\]
We will argue here through dispersive estimates and interpolation, following \cite{CastellaPerthame1996,keelTao1998}. The optimal result follows from a combination of these two arguments, which we integrate here for the sake of the reader. The first step is to obtain dispersive estimates.

\begin{lemma}[Dispersive estimates] If $p\in[1,\infty]$, then  
\[
    \|\rho(t)\|_{L^p_x} \lesssim t^{-d(1-\frac{1}{p})}\|f_0\|_{L^1_xL^{p}_v}
\]
and if $1 \leq r \leq p \leq \infty$, then
\begin{equation}
    \label{disp}
    \|f(t)\|_{L^p_xL^r_v} \lesssim t^{-d\left(\frac{1}{r}-\frac{1}{p}\right)} \|f_0\|_{L^r_xL^p_v}.
\end{equation}
\end{lemma}

\begin{proof}
This follows from complex interpolation between the two following observations. First, by Fubini's theorem we have $\|\rho(t, \cdot)\|_{L^1_x} \leq \|f_0\|_{L^1_{x,v}}$. We also have the dispersive estimate
\[
    |\rho(t,x)| \leq \int_{\mathbb{R}^d} |f_0(x-tv,v)|\dd v = t^{-d}\int_{\mathbb{R}^d} \left|f_0\left(y,\frac{x-y}{t}\right)\right| \dd y \leq t^{-d}\|f_0\|_{L^1_xL^{\infty}_v}.
\]
For the more general inequality, one can interpolate with the case $(r,p)=(p,p)$, for which the inequality holds with constant $1$, by conservation of the $L^p$ norms under Liouville's equation. 
\end{proof}

From these dispersive estimates, we will derive Strichartz estimates of the type $L^a_{x,v} \to L^q_t L^p_x L^r_v$. Specializing to the case $q=p=\frac{d+2}{d}$, $r=1$ and $a = \frac{d+2}{d+1}$ gives the desired result.

\begin{theorem} \label{thm.stri}
Let $(q,p,r,a)\in[1,\infty]^4$ such that the following conditions hold: 
\begin{equation}
\label{eq.conditions.strichartz}
    \frac{2}{q}=d\left(\frac{1}{r}-\frac{1}{p}\right), \qquad \frac{1}{a}=\frac{1}{2}\left(\frac{1}{r}+\frac{1}{p}\right), \qquad p\geq a, \quad \text{and } q>a.
\end{equation}
Then there holds
\begin{equation}
\label{eq.strichartz}
    \|f\|_{L^q_t(\mathbb{R};L^p_xL^r_v)} \lesssim \|f_0\|_{L^a(\mathbb{R}^{2d})}. 
\end{equation}
\end{theorem}

\begin{remark}
The conditions \eqref{eq.conditions.strichartz} are actually also necessary  for \eqref{eq.strichartz} to hold, see \cite{keelTao1998}. It is proven in \cite{GuoPeng2007} that there is no such estimate in the case $(d,q,p,a)=(1,2,\infty,2)$, but there is still an estimate if $L^{\infty}_x$ is replaced with ${\rm BMO}_x$. In the same article it is also shown that in all dimensions $d\geq 1$ the endpoint fails despite being conjectured in \cite{keelTao1998}.
\end{remark}

\begin{proof}[Proof of \Cref{thm.stri}] 
Assume first that the theorem has already been proven in the more restrictive case $q>2\geq a$. Now, let $(q,p,r,a)$ be as in the theorem but such that either $2\leq a<q$ or $a<q\leq 2$. Pick $\alpha>0$ such that $q\alpha := q_{\alpha} > 2 \geq \tilde a_{\alpha}=\alpha a$ and apply the restricted case to the function $f^{1/\alpha}$ solving the Liouville equation, and the exponents $(q_{\alpha},p_{\alpha},r_{\alpha},a_{\alpha})=(\alpha q, \alpha p, \alpha r, \alpha a)$. One can check that the conditions on these coefficients are met so that we have 
\[
    \|f^{\frac{1}{\alpha}}\|_{L^{q_{\alpha}}_tL^{p_{\alpha}}_xL^{r_{\alpha}}_v} \lesssim \|f_0^{\frac{1}{\alpha}}\|_{L^{a_{\alpha}}_{x,v}},
\]
which is nothing but the desired estimates.

We are left with the proof of the theorem in the case $q >2 \geq a$, which will appear as very convenient in the following, because of the application of Hardy--Littlewood--Sobolev's inequality and a square trick. 

Observe that the adjoint $S^*$ of the operator $S$ defined for each $t$ by $S(t)f_0(x,v)=f_0(x-tv,v)$ is given for all $\phi = \phi(t,x,v)$ by $S^*\phi(x,v)=\int_{\mathbb{R}}\phi(t,x+tv,v) \dd t = \int_{\mathbb{R}} S(-t)\phi(t)(x,v)\dd t$. 

Let's argue by duality: the boundedness of $S : L^a_{x,v} \to L^q_tL^p_xL^r_v$ is equivalent to that of its adjoint $S^* : L^{q'}_tL^{p'}_xL^{r'}_v \to L^{a'}_{x,v}$. since $a' \geq 2$, we can write, using a squaring trick:
\begin{align*}
    \|S^*\phi\|_{L^{a'}_{x,v}}^2 &= \left\| \iint \phi(t,x+tv,v)\phi(s,x+sv,v)\dd s \dd t \right\|_{L^{a'/2}_{x,v}} \\
    &= \left\| \iint \phi(t,x,v)\phi(s,x-(t-s)v,v)\dd s \dd t \right\|_{L^{a'/2}_{x,v}}.
\end{align*}
In the second equality we have used the change of variable $(x,v) \leftarrow (x-tv,v)$. 

An application of Young's integral inequality (since $\frac{a'}{2} \geq 1$), followed by Hölder's inequality with $\frac{2}{a'} = \frac{1}{p'} + \frac{1}{r'} = \frac{1}{r'} + \frac{1}{p'}$ yields
\begin{align*}
     \|S^*\phi\|_{L^{a'}_{x,v}}^2  & \leq \iint \|\phi(t)\|_{L^{p'}_xL^{r'}_v}\|S(t-s)\phi(s)\|_{L^{r'}_xL^{p'}_v}\dd s \dd t  \\ 
     &\leq \iint \|\phi(t)\|_{L^{p'}_xL^{r'}_v} |t-s|^{-\delta}\|\phi(s)\|_{L^{p'}_xL^{r'}_v}\dd s \dd t,
\end{align*}
with $\delta = d\left(\frac{1}{p'}-\frac{1}{r'}\right)=d\left(\frac{1}{r}-\frac{1}{p}\right)=\frac{2}{q} <1$. In the last inequality we have used the dispersion inequality \eqref{disp} applied to the couple $(r',p')$ (instead of $(p,r)$, and note that we have $1 \leq p'\leq r' \leq \infty$). 

The conclusion follows from an application of the Hardy--Littlewood--Sobolev inequality.
\end{proof}

\section{The Euclidean torus in dimension \texorpdfstring{$1$}{1}}

The aim of this section is to prove \Cref{thm.torus.1d}.

\subsection{Necessary condition}

We aim at proving that the inequality
\begin{equation}
\label{canardsiffleur}
\| \rho(t,x) \|_{L^{p}_{t,x}(\mathbb{R} \times \mathbb{T})} \lesssim \| f_0 \|_{L^{a}_v L^{b}_x},
\end{equation}
can only hold for all $f_0$ with zero average if
\[
    a = p', \quad b \geq \frac{p}{2} \quad \text{and} \quad p>2.
\]

\noindent \underline{The scaling argument.} We want to apply a classical scaling argument, but we need to be a bit careful, since the domain $\mathbb{T}$ cannot be scaled. We consider the data in $\mathbb{R}^d$
\[
    f_0(x,v) = \varphi(\lambda x) \psi(\mu v),
\]
where $\varphi,\psi \in \mathcal{C}^\infty_0$ and $\int \varphi(x) \dd x = 0$. The associated density is
\[
    \rho(t,x) = \int \varphi(\lambda(x-tv)) \psi(\mu v) \dd v = \mu^{-1} R(\lambda \mu^{-1} t, \lambda x),
\]
where
\[
    R(t,x) = \int \varphi(x-tv) \psi(v) \dd v.
\]
For a given $t$ it follows that
\[
    \operatorname{Supp} \rho(t,\cdot) \subset B\left( 0,C\left( \frac t \mu + \frac 1 \lambda \right) \right) \subset B\left(0,\frac 14 \right) \qquad \text{if} \;\; 
    \begin{cases}
        0 < t \ll \mu    \\ \lambda \gg 1.
    \end{cases}
\]
Under this condition, which we assume from now on, we can view $f_0(x,v)$ and $\rho(t,x)$ as functions on the torus. Then \eqref{canardsiffleur} implies that
\[
    \left\| \mu^{-1} R(\lambda \mu^{-1} t,\lambda x) \right\|_{L^p([0,c \mu] \times \mathbb{T})} \lesssim \left\|\varphi(\lambda x) \psi(\mu v) \right\|_{L^a_v L^b_x}
\]
which can also be written
\[
    \mu^{-1+\frac 1p} \lambda^{-\frac {2}p} \left\|  R(t,x) \right\|_{L^p([0, c \lambda] \times \mathbb{T})} \lesssim \lambda^{-\frac 1 b} \mu^{-\frac 1 a} \left\| \varphi(x) \psi(v) \right\|_{L^a_v L^b_x}.
\]
Considering the limits $\mu \to 0$, $\mu \to \infty$ and $\lambda \to \infty$ gives the conditions $a = p'$ and $b \geq \frac{p}{2}$. 

\medskip

\noindent
\underline{A function with a single frequency in $x$.} Choosing
\[
    f_0(x,v) = \varphi(v) e^{ix},
\]
the density then writes
\[
    \rho(t,x) = \widehat{\varphi}(t) e^{ix},
\]
and the inequality \eqref{canardsiffleur} becomes
\[
    \| \widehat{\varphi} \|_{L^p_t(\mathbb{R})} \lesssim \| \varphi \|_{L^a_v(\mathbb{R})},
\]
which is known to hold only if $p = a'$ and $p \geq 2$, see for instance \cite{Tao}. 

\medskip
\noindent
\underline{Why $p$ cannot be $2$.} There remains to exclude $p=2$, in which case the inequality would be
\[
    \|\rho \|_{L^2_{t,x}} \lesssim \| f_0 \|_{L^2_v L^1_x}.
\]
From \cref{lemmaL2} this would imply
\[
    \||\partial_x|^{-\frac 1 2} f_0 \|_{L^2_{x,v}} \lesssim \| f_0 \|_{L^2_v L^1_x}.
\]
Specializing to the case $f_0(x,v) = \varphi(x) \psi(v)$, this would mean that
\begin{equation}\label{sobolev.fail}
    \| |\partial_x|^{-\frac 1 2} \varphi \|_{L^2_{x}} \lesssim \| \varphi \|_{L^1_x},
\end{equation}
which is known to fail. Let's indeed consider $\varphi(x)=\chi_N(x)-1=N\chi(Nx)-1$ with $\chi \in \mathcal{C}_0^\infty$ of integral $1$. For $N$ sufficiently large, this function is well-defined on $\mathbb{T}$, has zero mean, and $\|\varphi\|_{L^{1}} \lesssim C < \infty$. Its Fourier coefficients are given by $\widehat{\chi_N}(k)=\widehat{\chi}\left(\frac{k}{N}\right)$, standing for the Fourier transform over $\mathbb{R}$ of $\chi$. Since $\widehat{\chi}(0) =1$, we have $|\widehat{\chi_N}(k)| > \frac{1}{2}$ if $|k| \leq \sqrt N$ and $N$ sufficiently large. Thus \eqref{sobolev.fail} would imply
\[
   \sum_{1 \leq |k|\leq \sqrt{N}} \frac{1}{|k|} \lesssim \sum_{|k|\geq 1} \frac{|\hat{\varphi}(k)|^2}{|k|} = \| |\partial_x|^{-\frac 1 2} \varphi \|_{L^2_{x}}^2 \lesssim \|\varphi\|_{L^1_x}^2 \leq C 
\]
which cannot hold as the left-hand side diverges as $N\to \infty$.

\subsection{The sufficient condition}
We aim at proving that the inequality
\begin{equation}
\label{canardsiffleur-bis}
    \| \rho(t,x) \|_{L^{p}_{t,x}(\mathbb{R} \times \mathbb{T})} \lesssim \| f_0 \|_{L^{a}_v L^{b}_x},
\end{equation}
holds for all $f_0$ with zero average if
\[
    a = p', \quad b \geq \frac{p}{2} \quad \text{and} \quad p>2.
\]
This result will be proved through a series of lemmas below. First, \Cref{lemmaLinf,lemmaL4,lemmaL2} provide estimates in $L^\infty_{t,x}$, $L^4_{t,x}$ and $L^2_{t,x}$ respectively. Second, these estimates are interpolated in \Cref{lemmainterp} to obtain all intermediate values of $p$.

\begin{lemma}[$L^\infty$ estimate]
\label{lemmaLinf}
For any $f_0 \in L^1_v L^\infty_x$ there holds 
\[
    \| \rho \|_{L^\infty_{t,x}(\mathbb{R} \times \mathbb{T})} \lesssim \| f_0 \|_{L^1_v L^\infty_x}.
\]
\end{lemma}

\begin{proof} 
This follows immediately from the formula expressing $\rho (t,x) = \int_{\mathbb{R}} f_0(x+tv,v)\dd v$, the triangle inequality and Minkowski's inequality. 
\end{proof}

\begin{lemma}[$L^4$ estimate]
\label{lemmaL4}
If $\int f_0(x,v) \dd x = 0$ for all $v$, then
\[
    \| \rho \|_{L^4_{t,x}(\mathbb{R} \times \mathbb{T})} \lesssim \| f_0 \|_{L^{\frac 43}_v L^2_x}.
\]
\end{lemma} 

\begin{proof} 
Expanding $f_0$ in Fourier series in $x$
\[
    f_0(x,v) = \sum_{k \in \mathbb{Z}^*}  a_k(v) e^{i k x},
\]
we have
\[
    \rho(t,x) = \int_{\mathbb{R}} f_0(x-tv,v) \dd v = \sum_{k\in\mathbb{Z}^*} \widehat{a_k}(kt) e^{i k x}.
\]
Using this decomposition and H\"older's inequality gives
\begin{align*}
    \| \rho \|_{L^4_{t,x}}^4 & = \sum_{k_1+k_2 = k_3 + k_4} \int \widehat{a_{k_1}}(k_1 t) \widehat{a_{k_2}}(k_2 t) \overline{\widehat{a_{k_3}}(k_3 t) \widehat{a_{k_4}}(k_4 t)} \dd t \\
    & \leq \sum_{k_1+k_2 = k_3 + k_4} \| \widehat{a_{k_1}}(k_1 t) \|_{L^4_t} \| \widehat{a_{k_2}}(k_2 t) \|_{L^4_t} \| \widehat{a_{k_3}}(k_3 t) \|_{L^4_t} \| \widehat{a_{k_4}}(k_4 t) \|_{L^4_t} \\
    & =  \sum_{k_1+k_2 = k_3 + k_4} |k_1 k_2 k_3 k_4|^{-\frac 14} \| \widehat{a_{k_1}}(\eta) \|_{L^4_\eta} \| \widehat{a_{k_2}}(\eta) \|_{L^4_\eta} \| \widehat{a_{k_3}}(\eta) \|_{L^4_\eta} \| \widehat{a_{k_4}}(\eta) \|_{L^4_\eta}.
\end{align*}
We now set $b_k = |k|^{-\frac 14} \| \widehat{a_k}(\eta) \|_{L^4_\eta}$ so that the above can be written
\[
    \| \rho \|_{L^4_{t,x}}^4 = \langle b_k , b_{-k} * b_k * b_k \rangle,
\]
where $\langle \,,\, \rangle$ stands for the $\ell^2$ inner product. On the one hand, the H\"older--O'Neil inequality, recalled in \eqref{holder-oneil}, gives
\[
    \| b_k \|_{\ell^{\frac 43,2}_k} \lesssim \left\| \| \widehat{a_{k}}(\eta) \|_{L^4_\eta} \right\|_{\ell^2_k} \quad \text{and} \quad
\| \rho \|_{L^4_{t,x}}^4 \lesssim \| b_k \|_{\ell^{\frac 43,2}_k} \| b_{-k} * b_k * b_k \|_{\ell^{4,2}_k}.
\]
On the other hand, the Young--O'Neil inequality, recalled in \eqref{eq.young-oneil} and the inclusion relation between Lorentz spaces gives
\[
    \| b_{-k} * b_k * b_k \|_{\ell^{4,2}_k} \lesssim \| b_k \|_{\ell^{\frac 43,2}_k} \| b_k * b_k \|_{\ell^{2,\infty}_k} \lesssim \| b_k \|_{\ell^{\frac 43,2}_k}^3
\]
Overall, this gives
\[
    \| \rho \|_{L^4_{t,x}}^4 \lesssim \| \widehat{a_{k}}(\eta) \|_{\ell^2_k L^{4}_\eta}^4.
\]
Finally, using successively the Hausdorff--Young, Minkowski and Parseval inequalities, we obtain that
\[
    \| \rho \|_{L^4_{t,x}} \lesssim  \| \widehat{a_k} \|_{\ell^2_k L^{4}_\eta} \lesssim \| a_k(v) \|_{\ell^2_k L^{\frac 43}_v} \leq \| a_k(v) \|_{L^{\frac 43}_v \ell^2_k} = \| f_0 \|_{L^{\frac 43}_v L^2_x}. \qedhere
\]
\end{proof}

\begin{lemma}[$L^2$ estimate]\label{lemmaL2}
Assume that $\int_{\mathbb{T}} f_0(x,v) \dd x = 0$ for all $v$, then
\[
    \| \rho \|_{L^2_{t,x}(\mathbb{R} \times \mathbb{T})} = \| |D_x|^{-\frac 12} f_0 \|_{L^2_{x,v}}.
\]
\end{lemma}

\begin{proof} 
Writing
\[
    f_0(x,v) = \sum_{k\in\mathbb{Z}^*} a_k(v) e^{ikx}, 
\]
we have
\[
    \rho(t,x) = \sum_{k \in \mathbb{Z}^*} \widehat{a_k}(kt) e^{ikx},
\]
and thus, by Parseval's and Plancherel's theorem
\[
    \| \rho \|_{L^2_{t,x}}^2 = \sum_{k\in\mathbb{Z}^*} |k|^{-1} \| \widehat{a_k} \|_{L^2}^2 = \| |D_x|^{-1/2} f_0 \|_{L^2}^2. \qedhere
\]
\end{proof}

\Cref{thm.torus.1d} will be a consequence of the following interpolation lemma, and its consequence on subspaces. More precisely, if $\Pi : f \mapsto f - \int_{\mathbb{T}}f(\cdot, x)\dd x$, then if $\Pi : B_i \to B_i$ for all $i\in\{0,1\}$ then interpolation commutes with this projection \cite[Theorem 1.17.1.1]{Triebel1978}
\[
    (\Pi B_0, \Pi B_1)_{\theta, r} = \Pi (B_0,B_1)_{\theta,r}.
\]

\begin{lemma}[Interpolation] \label{lemmainterp} 
Let $T$ be a linear operator defined from $\mathcal{C}^{\infty}(\mathbb{R}_v \times \mathbb{T}_x)$ to $\mathcal{C}^{\infty}(\mathbb{R}_{t}\times \mathbb{T}_{x})$ be a linear operator with the mapping properties: 
\begin{align}
    T : L^1_vL^{\infty}_x \longrightarrow L^{\infty}_{t,x}, \label{eq.Linft-bound}\\ 
    T :  L^{\frac{4}{3}}_vL^{2}_x \longrightarrow L^{4}_{t,x}, \label{eq.L4-bound} \\ 
    T :  L^2_v\dot H^{-\frac{1}{2}}_x \longrightarrow L^{2}_{t,x}. \label{eq.L2-bound}
\end{align}
Then for all $p \in (2,\infty]$, $T$ is a continuous operator $T : L^{p'}_vL^{\frac{p}{2}}_x \longrightarrow L^p_{t,x}$.
\end{lemma}

\begin{proof} 
Let $p \in (4, \infty)$ and $\theta = 1 - \frac{4}{p} \in (0,1)$, which is such that $L^p_{t,x}=[L^{4}_{t,x},L^{\infty}_{t,x}]_{\theta}$, the complex interpolation space \cite[Chapter 4]{BerghLofstrom}. Define $\frac{1}{p_\theta} = \frac{1-\theta}{4/3} + \theta = \frac{1}{p'}$.
From the complex interpolation theorem \cite[Theorem 4.1.2]{BerghLofstrom} we obtain 
\[
    T : [L^{\frac{4}{3}}_vL^{2}_x,L^1_vL^{\infty}_x]_{\theta} \to [L^{4}_{t,x},L^{\infty}_{t,x}]_{\theta}=L^p_{t,x}.
\]
We now make use of \cite[Theorem 5.1.2]{BerghLofstrom}, which yields 
\[
    [L^{\frac{4}{3}}_vL^{2}_x,L^1_vL^{\infty}_x]_{\theta}=L^{p(\theta)}_v([L^2_x,L^{\infty}_x]_{\theta})=L^{p'}_vL^{\frac{p}{2}}_x.
\]
Similarly, for any $p\in(2,4)$ and $\theta = 1-\frac{2}{p}$ we obtain the continuity: 
\[
    T : [L^2_v\dot H^{-\frac{1}{2}}_x,L^{\frac{4}{3}}_vL^2_x]_{\theta} \longrightarrow [L^2_{t,x},L^4_{t,x}]_{\theta}=L^p_{t,x}.
\]
Using again the property of interpolation of vector-valued $L^q$ spaces mentioned above and the complex interpolation of homogeneous Sobolev spaces there holds 
\[
    [\dot H_x^{-\frac{1}{2}},L^2_x]_\theta = \dot{H}^{\frac{1}{2}-\frac{2}{p}}_x. 
\]
The conclusion now follows from the fact that on the torus there holds $\dot{H}^{\frac{1}{2}-\frac{2}{p}}_x=H^{\frac{1}{2}-\frac{2}{p}}_x$, and the Sobolev embedding $L^{\frac{p}{2}}_x \hookrightarrow H^{\frac{1}{2}-\frac{2}{p}}_x$.
\end{proof}

\section{The Euclidean torus in dimension \texorpdfstring{$d \geq 2$}{>=}}

Our aim in this section is to prove \Cref{thm.noest-dim2}.

\subsection{Necessary conditions}

We shall prove here the necessary conditions in \Cref{thm.noest-dim2} for the inequality
\begin{equation}\label{gamma-p-a-b-ineq}
    \|\rho\|_{L^p_{t,x}} \lesssim \||v|^{\gamma}f_0\|_{L^a_vL^b_x},
\end{equation}
to hold for all $f_0$ with mean zero in $x$.

\medskip
\noindent\underline{The conditions $\gamma = d(1-\frac{1}{a})-\frac{1}{p}$ and $p\leq b \frac{d+1}{d}$.} They are obtained through a scaling argument as follows. Data of the form 
\[
    f_0(x,v) = \varphi(\lambda x) \psi(\mu v), \qquad \varphi, \psi \in \mathcal{C}_0^\infty(\mathbb{R}^d), \qquad \int \varphi = 0
\]
give $\rho(t,x)=\mu^{-d}R(\lambda\mu^{-1}t,\lambda x)$ for a function $R$. As long as $0<t\ll \mu$ and $\lambda \gg 1$, this can be viewed as a function on $\mathbb{T}^d$. The inequality becomes
\[
    \mu^{-d+\frac{1}{p}}\lambda^{-\frac{d+1}{p}}\|R\|_{L^p([0,c\lambda ] \times \mathbb{T}^d)} \lesssim \lambda^{-\frac{d}{b}}\mu^{-\gamma - \frac{d}{a}} \|\varphi(x)\psi(v)\|_{L^a_vL^b_x},
\]
which gives $\gamma = d(1-\frac{1}{a})-\frac{1}{p}$ and $p\leq b \frac{d+1}{d}$ after letting $\mu \to \pm \infty$ and $\lambda \to + \infty$.

\medskip
\noindent\underline{The condition $a \geq p'$.} It is obtained by examining the case of single mode data of the form
\begin{equation}
    \label{singlemodeansatz}
    f_0(x,v)=\varphi_1(v_1) \varphi_2(v') e^{ix_1},
\end{equation}
where $v_1 \in \mathbb{R}$ and $v' \in \mathbb{R}^{d-1}$ make up $v=(v_1,v')$.
Then \eqref{withweight} becomes
\begin{equation}
    \label{newineq}
    \widehat{\varphi_2}(0) \|\widehat{\varphi_1}\|_{L^p} \lesssim \||v|^{\gamma} \varphi_1(v_1) \varphi_2(v')\|_{L^a_v}.
\end{equation}
Choosing $\varphi_1(v_1) = \Phi(\lambda v_1)$ and $\varphi_2$ localized near the sphere $|v'| = 1$ such that $\widehat{\varphi}_2(0)\neq 0$, we obtain that $\varphi = \varphi_1 \varphi_2$ is localized near $|v|=1$ as $\lambda \to \infty$. The inequality becomes
\[
    \lambda^{-1+\frac 1p} \widehat{\varphi_2}(0) \|\widehat{\Phi}\|_{L^p} \lesssim \lambda^{-\frac 1 a} \| \Phi \|_{L^a} \| \varphi_2 \|_{L^a},
\]
which gives the desired condition as $\lambda \to \infty$.

\medskip
\noindent
\underline{The condition $a\neq p'$.} It relies on a refinement of the previous example: we choose now
\[
    \varphi_1 \in \mathcal{C}^\infty_0, \qquad \operatorname{supp} \varphi_1 \subset \left\{ \frac{1}{2}\leq |v_1| \leq 1\right\},  \qquad \varphi_2(v') = \langle v' \rangle^{-(d-1)} \mathbf{1}_{2 \leq |v'|<N}.
\]
Observe first that
\[
    \widehat{\varphi_2}(0) = \int \varphi_2 \sim \log N.
\]
Second, the condition $a=p'$ gives $\gamma = \frac{d-1}{p}$ and
\begin{align*}
    \||v|^{\gamma} \varphi_1(v_1) \varphi_2(v')\|_{L^a_v} & \sim \| \langle v' \rangle^{\gamma} \varphi_2(v') \|_{L^a_{v'}} \sim \| \langle v' \rangle^{\gamma -(d-1)} \mathbf{1}_{|v'| \lesssim N} \|_{L^{a}_{v'}} \\
    & \sim \| \langle v' \rangle^{-\frac{d-1}{p'}}  \mathbf{1}_{2 \leq |v'| \leq N} \|_{L^{p'}} \sim (\log N)^{\frac{1}{p'}}.
\end{align*}
Comparing the two above equivalents and \eqref{newineq} gives $\log N \lesssim (\log N)^{\frac{1}{p'}}$, which excludes $p < \infty$ as $N \to \infty$.

\medskip
\noindent\underline{The condition $a\leq p$.}
Consider a single mode zero-average initial data of the form 
\[
    f_0(x,v)=e^{ix_1}\psi_R(v),
\]
where $x=(x_1,x')\in \mathbb{T} \times \mathbb{T}^{d-1}$ and where $\psi_R$ is defined for all $v=(v_1,v') \in \mathbb{R} \times\mathbb{R}^{d-1}$ by 
\[
    \psi_R(v_1,v') = \frac{\mathbf{1}_{1<v_1<R}\mathbf{1}_{|v'|<v_1}}{v_1^{\frac{1}{p'}+(d-1)}}.
\]
We begin by computing for $t >0$
\[
    \rho(t,x)=e^{ix_1}\int_{\mathbb{R}^d} e^{-itv_1}\psi_R(v)\mathrm{d}v = e^{ix_1}\int_{1}^R e^{-itv_1}v_1^{-\frac{1}{p'}}\mathrm{d}v_1 = e^{ix_1}t^{-\frac{1}{p}}\int_{t}^{tR}e^{-i\tau}\tau^{-\frac{1}{p'}}\dd \tau.
\]
It follows that
\[
    \|\rho\|_{L^p_{t,x}}^p \geq \int_{0}^{\infty} \frac{1}{t} \left\vert \int_{t}^{tR}e^{-i\tau}\tau^{-\frac{1}{p'}}\dd \tau \right\vert^p \dd t.
\]
The function $H(a,b)= \int_{a}^{b}e^{-i\tau}\tau^{-\frac{1}{p'}}\dd \tau$ converges as $(a,b) \to (0,\infty)$ to some non-zero complex number. Therefore, there exists $\delta$, $M$ and $\varepsilon$ such that $|H(a,b)|>\varepsilon$ for $a < \delta$ and $b>M$. Then, if $R>\frac M \delta$ and $t \in [\frac{M}{R},\delta ]$, we have $|H(t,tR)|>\varepsilon$ and therefore
\begin{equation}
    \label{eq.LHS-bound}
    \|\rho\|_{L^p_{t,x}}^p \gtrsim \varepsilon^p\int_{\frac{M}{R}}^{\delta} \frac{\dd t}{t} \sim \log R. 
\end{equation}
Next, we write 
\[
    \||v|^{\gamma}f_0(x,v)\|_{L^a_vL^{b}_x}^a \simeq \int_{1}^R\int_{|v'|<v_1} (v_1^2+(v')^2)^{\frac{\gamma a}{2}}v_1^{-a(d-1)-\frac{a}{p'}} \dd v' \dd v_1.
\]
Using that on the support of the integral there holds $(v')^2 \leq v_1^2$ and $\gamma a - a(d-1)-\frac{a}{p'} = -d$ we infer 
\begin{equation}
    \label{eq.RHS-bound}
     \||v|^{\gamma}f_0(x,v)\|_{L^a_vL^{b}_x}^a \lesssim \int_1^R \frac{\dd v_1}{v_1} \sim \log R.
\end{equation}
In view of \eqref{eq.LHS-bound} and \eqref{eq.RHS-bound}, it must hold $(\log R)^{\frac{1}{p}} \lesssim (\log R)^{\frac{1}{a}}$, which implies $a \leq p$.

The above is valid for finite $p$. Assume $p=\infty$ and consider the same function. The right-hand side is still estimated by $(\log R)^{\frac{1}{a}}$, and for the left-hand side we now have: 
\[
    \|\rho\|_{L^{\infty}_{t,x}} \geq |\rho(0,0)|=\int_1^R \frac{\dd v_1}{v_1} = \log R.
\]
Therefore we must have $\log R \lesssim (\log R)^{\frac{1}{a}}$ which forces $a\leq 1$ and therefore $a=1$.

\medskip

\noindent \underline{The condition $p \geq 2$.} It is a consequence of $p' \leq a \leq p$.

\subsection{Sufficient conditions} We shall prove here the sufficient conditions in \Cref{thm.noest-dim2} for the inequality
\begin{equation}\label{gamma-p-a-b-ineq-suff}
    \|\rho\|_{L^p_{t,x}} \lesssim \||v|^{\gamma}f_0\|_{L^a_vL^b_x},
\end{equation}
to hold for all $f_0$ with zero-mean in $x$.
We start with a single mode estimate.

\begin{lemma}[Single mode estimate]\label{lem.mono-mode}
Let $2<p<\infty$, $p'<a\leq p$, and $\gamma=\frac{d}{a'}-\frac{1}{p}$. 
Let $k\in\mathbb Z^d\setminus\{0\}$, and let $f(x,v)=f_k(v)e^{ik\cdot x}$. Then there holds $\rho(f)(t,x)=\sqrt{2\pi}\rho_k(t)e^{ik\cdot x}$ and
\[
    \|\rho_k\|_{L^p(\mathbb R_t)} \lesssim |k|^{-\frac{1}{p}} \||v|^\gamma f_k\|_{L^a_v(\mathbb R^d)}.
\]
\end{lemma}

\begin{proof} 
Write $e_k=\frac{k}{|k|}$. We have
\[
    \rho(f)(t,x) = e^{ik\cdot x}\int_{\mathbb R^d}f_k(v)e^{-itk\cdot v}dv = \sqrt{2\pi} e^{ik\cdot x}\widehat g_k(t|k|e_k),
\]
so that $\rho_k(t)=\widehat{f}_k(t|k|e_k)$. Therefore
\[
    \|\rho_k\|_{L^p_{t}(\mathbb{R})} \sim |k|^{-\frac{1}{p}} \|\widehat f_k(te_k)\|_{L^p_t(\mathbb R)}.
\]
It is therefore enough to prove that for all smooth functions $g : \mathbb{R}^d \to \mathbb{R}$ there holds 
\[
    \|\widehat g(t e)\|_{L^p_t(\mathbb R)} \lesssim \||v|^\gamma g\|_{L^a_v(\mathbb R^d)} \qquad \mbox{if $e \in \mathbb{S}^{d-1}$.}
\]
By rotation invariance, assume $e=(1,0, \dots, 0)$. Our task is therefore reduced to proving that for any function $g : \mathbb{R}^d \to \mathbb{R}$ there holds
\begin{equation}
    \label{eq.target-mono}
    \|\widehat{g}(t, 0)\|_{L^{p}_t(\mathbb{R})} \lesssim \||v|^{\gamma}g\|_{L^a_v(\mathbb{R}^d)}.
\end{equation}
By duality, \eqref{eq.target-mono} is implied by 
\begin{equation}
    \label{eq.target-dual}
    \||w| ^{-\gamma}\widehat{h}(w_1)\|_{L^{a'}(\mathbb{R}^d_w)} \lesssim \|h\|_{L^{p'}(\mathbb{R})}.
\end{equation}  
Assume indeed that \eqref{eq.target-dual} holds. Since
\[
    \int_{\mathbb{R}} \widehat{g}(t,0)h(t)\mathrm{d}t = \frac{1}{(2\pi)^{\frac{d}{2}}} \int_{\mathbb{R}}\int_{\mathbb{R}^d}e^{-itw_1}g(w)h(t)\dd w \dd t = \frac{1}{(2\pi)^{\frac{d-1}{2}}} \int_{\mathbb{R}} |w| ^{\gamma}g(w)|w|^{-\gamma}\widehat{h}(w_1) \dd w,
\]
the inequality \eqref{eq.target-mono} now follows from the Hölder inequality, \eqref{eq.target-dual} and the dual characterization of $L^p(\mathbb{R})$.

Next, we write
\[
    \||w|^{-\gamma}\widehat{h}(w_1)\|_{L^{a'}(\mathbb{R}^d)}^{a'} = \int_{\mathbb{R}} |\widehat{h}(w_1)|^{a'} \int_{\mathbb{R}^{d-1}} \left(w_1^2+|w'|^2\right)^{-\frac{\gamma a'}{2}}\mathrm{d}w' \mathrm{d}w_1,
\]
where $w=(w_1,w')$, $w' \in \mathbb{R}^{d-1}$. Since $a>p'$ we deduce $\gamma a' = d - \frac{a'}{p} >d-1$ and therefore for $w_1\neq 0$ we have
\[
    \int_{\mathbb{R}} \left(w_1^2+|w'|^2\right)^{-\frac{\gamma a'}{2}}\mathrm{d}w' \sim | w_1|^{(d-1)-\gamma a'}. 
\]
We infer that
\[
    \||w|^{-\gamma}\widehat{h}(w_1)\|_{L^{a'}(\mathbb{R}^d_w)} \lesssim \||z|^{-\beta} \widehat{h}(z)\|_{L^{a'}(\mathbb{R})}, 
\]
where $\beta = \frac{1}{a'} - \frac{1}{p} \in (0,1]$. We can use the Hölder--O'Neil inequality \eqref{holder-oneil} to bound 
\[
    \||z|^{-\beta} \widehat{h}(z)\|_{L^{a'}(\mathbb{R})} =\||z|^{-\beta} \widehat{h}(z)\|_{L^{a',a'}(\mathbb{R})}\lesssim \||z|^{-\beta}\|_{L^{\frac{1}{\beta},\infty}}\|\widehat{h}\|_{L^{p,a'}} \lesssim \|\widehat{h}\|_{L^{p,a'}}. 
\]
Because $a'\geq p'$ it follows that 
\[
    \|\widehat{h}\|_{L^{p,a'}} \lesssim \|\widehat{h}\|_{L^{p,p'}} \lesssim \|h\|_{L^{p'}},
\]
where the last inequality stems from an application of the real-interpolation version of the classical Hausdorff--Young inequality, see \eqref{eq.hausdorffYoung-real}.
\end{proof}

\begin{proposition}\label{prop.horizontal-line}
Let $p\in (2,\frac{2(d+1)}{d}]$. Then for $\gamma = \frac{d}{a'}-\frac{1}{p}$, $a=2$ and $b\geq \frac{dp}{d+1}$, the estimate \eqref{eq.homogenous-d} holds.
\end{proposition}

\begin{proof} Write 
\[
  f_0(x,v)=\sum_{k\in\mathbb{Z}^d \setminus \{0\}} f_k(v)e^{ik\cdot x}, \qquad \rho(t,x) =\sqrt{2\pi} \sum_{k\in\mathbb{Z}^d \setminus \{0\}} \rho_k(t) e^{ik\cdot x},
\]
with $\rho_k(t)=\widehat{f}_k(tk)$. By virtue of the Sobolev embedding theorem $\dot{H}^s(\mathbb{R}^d) \hookrightarrow L^p(\mathbb{R}^d)$, we can bound for any $t \in \mathbb{R}$ 
\[
    \|\rho(t)\|_{L^p_x(\mathbb{R}^d)} \lesssim \|\rho(t,\cdot)\|_{\dot{H}^s(\mathbb{R}^d} \lesssim \left( \sum_{k\in\mathbb{Z}^d \setminus \{0\}} |k|^{2s}|\rho_k(t)|^2\right)^{\frac{1}{2}}, \qquad s = \frac{d}{2} - \frac{d}{p}.
\]    

Since $p>2$, an application of the Minkowski inequality followed by \Cref{lem.mono-mode} applied with $a=2$ and $\gamma=\frac{d}{a'}-\frac{
1}{p}$ yields
\[
    \|\rho\|_{L^p_{t,x}} \lesssim \left( \sum_{k\in\mathbb{Z}^d \setminus \{0\}} |k|^{2s}\|\rho_k\|_{L^p(\mathbb{R})}^2\right)^{\frac{1}{2}} \lesssim 
    \left(\sum_{k\in\mathbb{Z}^d \setminus \{0\}} |k|^{2s-\frac{2}{p}}\||v|^{\gamma}f_k\|^2_{L^2_v}\right)^{\frac{1}{2}} = \||D_x|^{s-\frac{1}{p}}|v|^{\gamma}f_0\|_{L^2_{x,v}}.
\]
Observe that $s-\frac{1}{p} = \frac{d}{2} - \frac{d+1}{p} \leq 0$ whenever $p \leq \frac{2(d+1)}{d}$, which is our assumption. Therefore, a final application of the Sobolev embedding theorem gives for all $b\geq \frac{dp}{d+1}$,
\[
    \|\rho\|_{L^p_{t,x}} \lesssim \||v|^{\gamma}f_0\|_{L^2_vL^b_x}. \qedhere
\]
\end{proof}

\begin{proof}[Proof of \Cref{thm.noest-dim2}~(ii)] 
Assume that $(p,a,b)$ satisfy
\[
    2<p<\infty, \quad p'<a \leq a_*(p) := \min\left\{2,\frac{dp}{dp-(d+1)}\right\},\quad b\geq b_*(p):=\frac{dp}{d+1}.
\]
The theorem follows from interpolation between the estimate: 
\[
    \|\rho\|_{L^{\infty}_{t,x}} \lesssim \|f_0\|_{L^1_vL^{\infty}_x},
\]
and the estimates from \Cref{prop.horizontal-line}: 
\[
    \|\rho\|_{L^{q}_{t,x}} \lesssim \||v|^{d-\frac{2}{q}}f_0\|_{L^2_vL^{b_*(q)}_x}, \qquad 2 < q \leq \frac{2(d+1)}{d}.
\]
In order to be precise, set $\theta = \frac{2}{a'} \in (0,1)$. Take $(p_1,a_1,b_1)=(\frac{2p}{a'},2,b_*(p_1))$,  $\frac{1}{p_1'}=\frac{a'}{2p'}-\frac{a'}{2}+1$. We can check indeed that $p_1 \leq \frac{2(d+1)}{d}$.
Set $(p_0,a_0,b_0)=(\infty,1,\infty)$. It follows from \Cref{prop.interpolation}~(i) that:\footnote{Strictly speaking we are interpolating the zero-mean in $x$ subspace of these spaces, which follows since the projection onto this subspace is continuous in both endpoint norms, see \cite[Theorem 1.17.1.1]{Triebel1978}.}
\[
    \left[L^1(dv;L^{\infty}_x),L^2\left(|v|^{d-\frac{2}{p_1}}dv,L^{b_*(p_1)}\right)\right]_{\theta} = L^{a_{\theta}}(|v|^{a_{\theta}\gamma_{\theta}}dv;L^{b_{\theta}}), 
\]
so that the real interpolation theorem implies the bound 
\[
    \|\rho\|_{L^{p_{\theta}}_{t,x}} \lesssim \||v|^{\gamma_{\theta}} f_0\|_{L^{a_{\theta}}L^{b_{\theta}}}.
\]
It remains to observe that $\frac{1}{p_{\theta}} = \frac{1-\theta}{\infty} + \frac{\theta}{p_1} = \frac{1}{p}$, $\frac{1}{a_{\theta}} = \frac{1-\theta}{a_0} + \frac{\theta}{a_1} = \frac{1}{a}$, $\frac{1}{b_{\theta}} = \frac{1-\theta}{\infty} + \frac{\theta}{b_*(p_1)} = \frac{1}{b_*(p)}$ and $\gamma_{\theta}=\frac{d}{a'}-\frac{1}{p}=\gamma$.
\end{proof}

\begin{proposition}\label{prop.substitute}
The estimate \eqref{eq.homogenous-d} holds for $(p,a,b)$ in the following cases:
\begin{enumerate}[label=(\roman*)]
    \item if $L^{a}$ is replaced with the Lorentz space $L^{a,r}$, $r<a_*(p)$;
    \item if $b>\frac{dp}{d+1}$ and $p'<a<p$;
    \item if $|v|^{\gamma}$ is replaced with $|v|^{\gamma}h_{\varepsilon}(v)$, with $h_{\varepsilon}(v)=(|v|^{-\varepsilon}\mathbf{1}_{|v|<1} + |v|^{\varepsilon}\mathbf{1}_{|v|>1})$. 
\end{enumerate}
\end{proposition}

\begin{proof} 
In the following write $\gamma(p,a)=\frac{d}{a'}-\frac{1}{p}$.

\medskip
\noindent \textit{(i)} We only have to consider the case $a>a_*(p)$. Start with $r$ such that $p'<r \leq a_*(p)$ and write 
\[
    \|\rho\|_{L^p_{t,x}} \lesssim \||v|^{\gamma(p,r)}f_0\|_{L^r_vL^{b}_{x}}
\]
Since $\gamma(p,a)-\gamma(p,r)=d\left(\frac{1}{r}-\frac{1}{a}\right) = \frac{d}{s}$ for some $s>0$, we have $|\cdot|^{-(\gamma(p,a)-\gamma(p,r))} \in L^{s,\infty}(\mathbb{R}^d)$. 
It remains to apply the Hölder--O'Neil inequality \eqref{holder-oneil} in order to obtain
\[
    \|\rho\|_{L^p_{t,x}} \lesssim \||v|^{\gamma(p,a)}f_0\|_{L^{a,r}_vL^b_x},
\]
which is the targeted estimate.

\medskip 
\noindent \textit{(ii)} 
Fix $p'<a<p$ and $b>b_*(p)$. Choose $p_0<p<p_1$ sufficiently close to $p$ such that $ b\ge b_*(p_i)$ which is possible because $b>b_*(p)=\frac{dp}{d+1}$ is a continuous increasing function of $p$. Let $\theta\in(0,1)$ such that $\frac{1}{p}=\frac{1-\theta}{p_0}+\frac{\theta}{p_1}$. Observe now that it is possible to choose $a_0$ and $a_1$ in order to satisfy $\gamma(p_i,a_i)=\gamma$. Since the $p_i$ can be chosen arbitrarily close to $p$, it follows that we can ensure $p_i'<a_i<p_i$. Also, observe that we have $\frac{1-\theta}{a_0} + \frac{\theta}{a_1} = \frac{1}{a}$.

The Lorentz substitute estimates from \textit{(i)} that we use are the following bounds 
\[
    \mathcal{L}^{a_i,1}_{|v|^{\gamma}}(dv;L^b_x) \longrightarrow L^{p_i}_{t,x}, \qquad i \in \{0, 1 \},
\]
so that the real interpolation theorem gives the continuity: 
\[
    \left( \mathcal{L}^{a_0,1}_{|v|^{\gamma}}(dv;L^b_x), \mathcal{L}^{a_1,1}_{|v|^{\gamma}}(dv;L^b_x) \right)_{\theta,a} \longrightarrow (L^{p_0}_{t,x},L^{p_1}_{t,x})_{\theta,a}.
\]
It remains to observe that $(L^{p_0}_{t,x},L^{p_1}_{t,x})_{\theta,a}=L^{p,a}_{t,x} \hookrightarrow L^p_{t,x}$ because $a \leq p$, and 
\[ 
    \left( \mathcal{L}^{a_0,1}_{|v|^{\gamma}}(dv;L^b_x), \mathcal{L}^{a_1,1}_{|v|^{\gamma}}(dv;L^b_x) \right)_{\theta,a} = \mathcal{L}^{a,a}_{|v|^{\gamma}}(dv;L^b_x)=L^a(|v|^{a\gamma}dv,L^b_x),
\]
which stems from an application of \Cref{prop.interpolation}~(ii). 

\medskip
\noindent\textit{(iii)} Assume $a>a_*(p)$ and take $p'<\alpha< \min\{a,a_*(p)\}$. One can therefore write $\frac{1}{\alpha}-\frac{1}{a}=\frac{1}{s}$ for some $s>0$ and 
\[
    \gamma(p,a)-\gamma(p,\alpha) =d\left(\frac{1}{\alpha}-\frac{1}{a}\right) =\frac{d}{s}.
\]
Write $|v|^{\gamma(p,\alpha)}f = m_{\varepsilon}(v)|v|^{\gamma(p,a)}h_{\varepsilon}(v)f$ where $m_{\varepsilon}(v)=|v|^{\gamma(p,\alpha)-\gamma(p,a)}h_{\varepsilon}(v)^{-1}=|v|^{-\frac{d}{s}}h_{\varepsilon}(v)^{-1}$. For $|v|\ll 1$ there holds $|m_{\varepsilon}(v)| \lesssim |v|^{-\frac{d}{s}+\varepsilon}$, and for $|v| \gg 1$ there holds $|m_{\varepsilon}(v)| \lesssim |v|^{-\frac{d}{s}-\varepsilon}$, so that $m_{\varepsilon}\in L^s(\mathbb{R}^d)$. 

It remains to use the strong bound with $(p,\alpha)$ to bound 
\[
    \|\rho\|_{L^p_{t,x}} \lesssim
    \||v|^{\gamma(p,\alpha)}f\|_{L^\alpha_vL^b_x} \lesssim \|m_{\varepsilon}\|_{L^s_v}\||v|^{\gamma(p,a)}h_{\varepsilon}(v)f\|_{L^a_vL^b_x} \lesssim \||v|^{\gamma(p,a)}h_{\varepsilon}(v)f\|_{L^a_vL^b_x},
\]
where we have used the Hölder inequality.
\end{proof}

\section{The Kakeya problem on the cylinder}

\label{section_proof_Kakeya}

This section is devoted to the proof of \Cref{thm.Kakeya}.

\subsection{Proof of the inequality}

Write 
\[ 
    T_k = \left\{(t,x) \in\mathbb{R}^{d+1} : |t-t_k| \leq \frac{R}{2}, \quad  \operatorname{dist}_{\mathbb{T}^d}(x,x_k +(t-t_k)v_k) \leq \delta \right\}
\]
and $\mathcal{T}=\{T_k\}_{k\in\mathcal{K}} = \bigcup_{j\in\mathbb{Z}} \mathcal{T}_j$, for the collection of tubes, where 
\[
    \mathcal{T}_j = \{T=T(t_k,x_k) \in \mathcal{T} : t_k \in jR+[0,R))\}.
\]
Fix some $j\in \mathbb{Z}$ and reference the associated tubes with indices $k \in \{1, \dots, N_j\}$. Note that $\sum_{j \in \mathbb{Z}} N_j = \# \mathcal{T} \lesssim \left(\frac{R}{\delta}\right)^d$, in view of the separation condition.

Arguing by duality, we choose $g \in L^{p'}(\mathbb{R} \times \mathbb{T}^d)$, $g\geq 0$ and start writing 
\begin{align*}
    Z_j(g):=\int_{\mathbb{R} \times \mathbb{T}^d} \left( \sum_{k=1}^{N_j} \mathbf{1}_{T_k}(t,x) \right) g(t,x) \dd x \dd t = \sum_{k=1}^{N_j} \int_{T_k} g(t,x) \dd x \dd t  \lesssim \sum_{k=1}^{N_j} \int_{B(x_k,\delta)} Xg(x,v_k) \dd x,
\end{align*}
where we have used the definition of the $X$-ray transform. The $X$-ray transform and the operator $\tilde{X}$ are comparable because we take $|v| \sim 1$. After translating the tubes so that they intersect $[-R,R]$ in $t$, and therefore are all supported in $[-2R,2R]$ in $t$. We can replace $X(x,v_k)$ with their averages of $X(x,v)$ for $v$ in a small ball:\footnote{This is because $\int_{B(x_i,\delta)}g(t,x+tv_k)\dd x \lesssim \int_{B(x_i,C\delta)}g(t,x+tv_k)\dd x$ for $v \in B\left(v_k, c\frac{\delta}{R}\right)$}
\begin{align*}
  Z_j(g) &\lesssim \left( \frac{R}{\delta} \right)^d  \sum_{k=1}^{N_j} \int_{B(x_k,\delta) \times B(v_k, \frac{\delta}{R})} [\widetilde{X}g] (x,v) \dd x \dd v \\
    & = \left( \frac{R}{\delta} \right)^d  \sum_{k=1}^{N_j} \int_{B(x_k,\delta) \times B(v_k, \frac{\delta}{R})} [\Pi_0 \widetilde{X}g] \dd x \dd v + \left( \frac{R}{\delta} \right)^d  \sum_{k=1}^{N_j} \int_{B(x_k,\delta) \times B(v_k, \frac{\delta}{R})} [\Pi_{\neq} \widetilde{X}g] \dd x \dd v \\
    & =: {\rm I}_j + {\rm II}_j.  
\end{align*}
It remains to estimate each of the two terms ${\rm I}_j,{\rm II}_j$ on the right-hand side. 

We first notice that
\[
    [\Pi_0 \widetilde X g ](x,v) = \frac{1}{(2\pi)^d} \int g(t,x+tv) \dd t \dd x \lesssim \| g \|_{L^1(\mathbb{R}^{d+1})} \qquad \text{for all } (x,v) \in \mathbb{T}^d \times \mathbb{R}^d, |v| \sim 1.
\]
Therefore,
\begin{align*}
|{\rm I}_j| & \lesssim \left( \frac{R}{\delta} \right)^d \| \Pi_0 \widetilde X g \|_{L^\infty} \left| \bigcup_{k=1}^{N_j} B(x_k,\delta) \times B\left(v_k, \frac{\delta}{R}\right) \right| \\
& \lesssim \left( \frac{R}{\delta} \right)^d N_j \delta^d \left(\frac{\delta}{R}\right)^d \| \Pi_0 \widetilde X g \|_{L^\infty_{x,v}} \\
&\lesssim N_j \delta^d \| g \|_{L^1_{t,x}} \lesssim N_j\delta^d R^{\frac{1}{p}} \| g \|_{L^{p'}_{t,x}},
\end{align*}
where the last inequality is a consequence of H\"older's inequality.

We turn to the term ${\rm II}_j$ that involves $\Pi_{\neq}$. By Hölder's inequality in $x$ we have:
\[
    |{\rm II}_j |  \lesssim \left( \frac{R}{\delta} \right)^d \delta^{\frac d b} \int_{\bigcup_{k=1}^{N_j}B(v_k,\frac \delta R)}  \| \Pi_{\neq} \widetilde{X} g \|_{L^{b'}_x} \dd v. 
\]
Now, summing in $j$ the above, we obtain the bound: 
\[
    \int_{\mathbb{R} \times \mathbb{T}^d} \left( \sum_{k \in \mathcal{K}} \mathbf{1}_{T_k}(t,x) \right) g(t,x) \dd x \dd t \lesssim \sum_{j\in \mathbb{Z}} N_j\delta^d R^{\frac{1}{p}} \| g \|_{L^{p'}_{t,x}} +  \left( \frac{R}{\delta} \right)^d \delta^{\frac d b} \int_{\bigcup_{k \in \mathcal{K}} B(v_k,\frac \delta R)}  \| \Pi_{\neq} \widetilde{X} g \|_{L^{b'}_x} \dd v. 
\]
Remark that
\[
    \left\vert \bigcup_{k \in \mathcal{K}} B\left(v_k,\frac \delta R\right) \right\vert \leq \# \mathcal{T} \left(\frac{\delta}{R}\right)^d \lesssim 1.
\]
It then follows by Hölder's inequality in $v$, and also noting that $|v| \sim 1$ in the integration region, that: 
\begin{align*}
    \int_{\mathbb{R} \times \mathbb{T}^d} \left( \sum_{k\in \mathcal{K}} \mathbf{1}_{T_k}(t,x) \right) g(t,x) \dd x \dd t &\lesssim  R^{d+\frac{1}{p}} \| g \|_{L^{p'}_{t,x}} +  \left( \frac{R}{\delta} \right)^d \delta^{\frac d b} \| \Pi_{\neq} \langle v \rangle^{-\gamma}\widetilde{X} g \|_{L^{a'}_vL^{b'}_x} \\ 
    & \lesssim \left(  R^{d+\frac{1}{p}}  + R^d \delta^{\frac{d}{b}-d}\right) \| g \|_{L^{p'}_{t,x}},
\end{align*}
where the last step stems from an application of the dual estimate of \Cref{thm.torus.1d,thm.noest-dim2}. The result follows by choosing $b = \frac{dp}{d+1}$, which, for fixed $p$, is always possible for some $a$ by \Cref{thm.torus.1d,thm.noest-dim2}, see also \Cref{figue}. 

\subsection{Optimality of the inequality} Consider first the bush example: a maximal family of tubes $\{T_k\}_{k\in \mathcal{K}}$ which are $\frac \delta R$-separated in direction and intersect at one point. Such a maximal family has $\sim  \left( \frac{R}\delta \right)^d$ elements, and thus $\sum_{k\in \mathcal{K}} \mathbf{1}_{T_k}$ has size $\sim \left( \frac{R}\delta \right)^d$ on a ball of radius $\sim \delta$. Therefore,
\[
    \left\| \sum_{k\in \mathcal{K}} \mathbf{1}_{T_k} \right\|_{L^p_{t,x}} \gtrsim \left( \frac{R}\delta \right)^d \delta^{\frac{d+1}{p}} = R^d \delta^{\frac{d+1}{p}-d}.
\]

Furthermore, assuming that all tubes are contained in $[-CR,CR] \times \mathbb{T}^d$, we consider a maximal subfamily of $\frac{\delta}{R}$ separated tubes.  
On the one hand, by maximality,
\[
    \left\| \sum_{k\in \mathcal{K}} \mathbf{1}_{T_k} \right\|_{L^1} = \sum_k |T_k| \sim R^{d+1}.
\]
On the other hand, by H\"older's inequality, because such a family is included in $[-CR,CR] \times \mathbb{T}^d$ which has measure $\sim R$, we have
\[
    \left\| \sum_{k\in \mathcal{K}} \mathbf{1}_{T_k} \right\|_{L^1} \lesssim R^{1-\frac{1}{p}}  \left\| \sum_k \mathbf{1}_{T_k} \right\|_{L^p}.
\]
Combining these two inequalities gives
\[
    \left\| \sum_{k\in \mathcal{K}} \mathbf{1}_{T_k} \right\|_{L^p} \gtrsim R^{d+\frac 1p}.
\]
Note that in contrast, in the Euclidean case such a maximal family can occupy a volume $\sim R^{d+1}$ compared with $R$ in the torus case.

\appendix

\section{Lorentz spaces}
We gather here a few results on Lorentz spaces which are used in the proofs, referring to \cite{BerghLofstrom,Visan} for complete presentations. We omit the underlying measure space which is unimportant; for us, it will be the Euclidean torus.

The Lorentz spaces $L^{p,q}$, $1 \leq p,q \leq \infty$, are Banach spaces which can be obtained by interpolation \cite[Theorem 5.3.1]{BerghLofstrom} between $L^1$ and $L^\infty$:
\begin{equation}
    \label{eq.interpolation-lorentz1}
    L^{p,q} = (L^1\,,\,L^\infty)_{\theta,q}, \qquad \frac{1}{p} = 1 - \theta. 
\end{equation}
The norm on $L^{p,q}$ is equivalent to the quasi-norm
\[
    \|f \|_{L^{p,q}}^* = \left\| \| f_n(x) \|_{L^p_x(\mathbb{R}^d)} \right\|_{\ell^q_n(\mathbb{Z})}, \qquad \qquad f_n(x) = f(x) \mathbf{1}_{2^n < f(x) < 2^{n+1}}.
\]
It follows from this expression that
\[
    \| f \|_{L^{p,p}} \sim \| f\|_{L^p}, \qquad \| f \|_{L^{p,q_1}} \lesssim \| f \|_{L^{p,q_2}} \qquad \mbox{if $q_1 \geq q_2$}.
\]
The Hölder inequality was generalized to the scale of Lorentz spaces by Hunt \cite{Hunt}, see also \cite{Oneil1963} and \cite[p. 73]{grafakos2011classical}. It states that

\begin{equation}
    \label{holder-oneil}
        \|fg\|_{L^{r,s}} \lesssim \|f\|_{L^{p,s_1}}\|g\|_{L^{q,s_2}},
\end{equation}
whenever
\[
 \frac{1}{r}=\frac{1}{p} + \frac{1}{q}, \qquad  \frac{1}{s} = \frac{1}{s_1} + \frac{1}{s_2}, \qquad 0<s_1,s_2 \leq \infty, \qquad  0 \leq p,q,r \leq \infty.
\]

The Young inequalities were generalized by O'~Neil \cite{Oneil1963}, see also \cite[p. 73]{grafakos2011classical}, under the form
\begin{equation}
    \label{eq.young-oneil}
    \|f*g\|_{L^{r,s}} \lesssim \|f\|_{L^{p,s_1}}\|g\|_{L^{q,s_2}}, 
\end{equation}
whenever
\[
    1+\frac{1}{r}=\frac{1}{p}+\frac{1}{q}, \qquad \frac{1}{s}=\frac{1}{s_1} + \frac{1}{s_2}, \qquad 1 < p, q, r \leq \infty, \qquad 0<s_1,s_2\leq \infty. 
\]

Finally, we will also resort to the following sharp version of the Hausdorff--Young inequality 
\begin{equation}
    \label{eq.hausdorffYoung-real}
    \|\widehat{f}\|_{L^{p,p'}} \lesssim \|f\|_{L^{p'}}, \qquad p \in [2,\infty).  
\end{equation}
Indeed, the Fourier transform $\mathcal{F}$ enjoys the following mapping properties
\[
    \mathcal{F} : L^2 \to L^2, \qquad \mathcal{F} : L^1 \to L^{\infty}.
\]
Interpolation via the real method $(\cdot,\cdot)_{\theta,p'}$ with $\theta = \frac{2}{p}$ and for $p \in (2,\infty)$ yields the boundedness of $\mathcal{F} : L^{p'}=L^{p',p'} \to L^{p,p'}$, which is the claimed inequality.

\section{Interpolation facts}

Let $X$ be a Banach space, and $\dd v$ the Lebesgue measure on $\mathbb{T}^d$, and $w : \mathbb{T}^d \to [0,\infty)$, which is positive almost everywhere. We recall that an equivalent norm on $L^{a,r}(w^a\dd v, X)$ is given by
\[
    \|f\|_{L^{a,r}(w^a\dd v, X)} = \left\|t \nu\left(\{v : \|f(v)\|_{X}>t\}\right)^{\frac{1}{a}} \right\|_{L^{r}(\mathbb{R}_+,\frac{\dd t}{t})}, \qquad \dd \nu = w(v)^a \dd v.
\]
This norm is in general not equivalent to the norm 
\[
    \|f\|_{\mathcal{L}_w^{a,r}(\dd v;X)} = \|wf\|_{L^{a,r}(\dd v;X)}.
\]
In the following, we denote by $\mathcal{L}_w^{a,r}(\dd v;X)$ the space obtained by completion under this norm and after identifying functions by equality almost everywhere.

Observe that when $r=a$ there holds
\[
    \mathcal{L}^{a,a}_{w}(\dd v;X)=L^a(w^a\dd v;X).
\]
Let us recall interpolation notations. 
Let $(A_0,A_1)$ be a compatible couple of Banach spaces, i.e., two
Banach spaces continuously embedded in a common space.

\medskip

\noindent \underline{The real method.} 
For $a\in A_0+A_1$ and $t>0$, the $K$-functional is defined by
\[
    K(t,a;A_0,A_1) := \inf\left\{ \|a_0\|_{A_0}+t\|a_1\|_{A_1}, a=a_0+a_1 \in A_0 + A_1 \right\},
\]
For $0<\theta<1$ and $1\leq q\leq\infty$, the real interpolation space
$(A_0,A_1)_{\theta,q}$ consists of all $a\in A_0+A_1$ such that
\[
    \|a\|_{(A_0,A_1)_{\theta,q}} := \left\| t^{-\theta}K(t,a;A_0,A_1)\right\|_{L^q((0,\infty),\frac{\dd t}{t})}
\]
is finite, with the usual modification when $q=\infty$.
\medskip 

\noindent \underline{The complex method.}
Consider the strip $S=\{z\in\mathbb{C} : 0\leq\operatorname{Re}z\leq1\}$.
The Calder\'on space $\mathcal{F}(A_0,A_1)$ consists of all functions
$f : \mathbb{C} \to A_0+A_1$ which are analytic on $S$, continuous on $\overline{S}$, and such that the following subsets are bounded: 
\[
    \{f(z), z \in S\} \subset A_0 + A_1, \qquad \{f(j+it) : t \in \mathbb{R}\} \subset A_j, \quad j\in\{0,1\}.
\]
It is endowed with the norm
\[
    \|f\|_{\mathcal{F}(A_0,A_1)}  := \max_{j \in \{0,1\}}\sup_{t\in\mathbb{R}}\|f(j+it)\|_{A_j}.
\]
For $0<\theta<1$, the complex interpolation space $[A_0,A_1]_\theta$ consists of all elements $a\in A_0+A_1$ for which there exists
$f\in\mathcal{F}(A_0,A_1)$ satisfying $f(\theta)=a$, with norm
\[
    \|a\|_{[A_0,A_1]_\theta} := \inf\{\|f\|_{\mathcal{F}(A_0,A_1)}, f\in\mathcal{F}(A_0,A_1), f(\theta)=a\}.
\]

\medskip

\noindent
\underline{The interpolation theorems.} Consider a linear operator $T : A_0 + A_1 \to B_0 + B_1$, where $(A_0,A_1)$ and $(B_0,B_1)$ are compatible couples. Assume also that $T : A_j \to B_j$ is continuous for all $j \in \{0,1\}$, then it follows \cite[Theorem 3.1.2, Theorem 4.1.2]{BerghLofstrom} that $T : (A_0,A_1)_{\theta,r} \to (B_0,B_1)_{\theta,r}$ is continuous for all $\theta \in (0,1)$ and $r \in [1, \infty]$, and also continuous as a map $[A_0,A_1]_{\theta} \to [B_0,B_1]_{\theta}$.

Our goal is to prove the following interpolation facts that are used in this article. 
These are not new results, but they are not easily found in the literature.

\begin{proposition}\label{prop.interpolation}
The following interpolation embeddings hold: 
\begin{enumerate}[label=(\roman*)]
    \item Let $a_0, a_1 \in [1,\infty)$ such that $a_0\neq a_1$, and $1 \leq b_0, b_1\leq \infty$. Let $w_0, w_1 : \mathbb{T} \to [0,\infty)$ be two measurable functions, taking positive values almost everywhere. Define for $\theta \in (0,1)$: 
    \[
        \frac{1}{a_{\theta}} = \frac{1-\theta}{a_0} + \frac{\theta}{a_1}, \qquad \frac{1}{b_{\theta}} = \frac{1-\theta}{b_0} + \frac{\theta}{b_1}, \qquad w_{\theta}=w_0^{\frac{(1-\theta)a_\theta}{a_0}}w_1^{\frac{\theta a_\theta}{a_1}}. 
    \]
    Then, 
    \[
        [L^{a_0}(w_0\dd v  ;L^{b_0}),L^{a_1}(w_1\dd v;L^{b_1})]_{\theta} = L^{a_{\theta}}(w_{\theta}\dd v;L^{b_{\theta}}).
    \]    
    \item Let $a_0, a_1 \in [1,\infty)$ such that $a_0 \neq a_1$. Let $\theta \in (0,1)$, and $1\leq r \leq \infty$. Define $a_{\theta}$ such that $\frac{1}{a_{\theta}} = \frac{1-\theta}{a_0} + \frac{\theta}{a_1}$. Let $(X,\|\cdot\|_X)$ be a Banach space. Let $w : \mathbb{T} \to [0,\infty)$ taking positive values almost everywhere. Then 
    \[
        \left(\mathcal{L}_w^{a_0,1}(\dd v;X),\mathcal{L}_w^{a_1,1}(\dd v;X)\right)_{\theta,r} = \mathcal{L}_w^{a_{\theta},r}(\dd v;X).
    \]
\end{enumerate}
\end{proposition}

\begin{proof}
(\textit{i}) Our starting point is the following interpolation result \cite[Theorem 5.1.2]{BerghLofstrom}: 
    \[
        [L^{a_0}(\dd v,A_0),L^{a_1}(\dd v,A_1)]_{\theta} = L^{a_{\theta}}(\dd v,[A_0,A_1]_{\theta}),
    \]
    for Banach spaces $A_0$ and $A_1$, for which we take $A_i=L^{b_i}$ in our case. Then we conclude as in \cite[Theorem 5.5.3]{BerghLofstrom} by observing that the map
    \[
        S : \mathcal{F}(L^{a_0}(\dd v ;A_0),L^{a_1}(\dd v ; A_1)) \longrightarrow \mathcal{F}(L^{a_0}(\dd v ;A_0),L^{a_1}(\dd v ; A_1))    
    \]
    defined by $Sf(z)=w_0^{\frac{1-z}{a_0}}w_1^{\frac{z}{a_1}} f(z)$, where $\mathcal{F}(X,Y)$ denotes the Calder\'on class defined in the complex interpolation method is an isometric isomorphism.
    
\medskip 
(\textit{ii}) Since the multiplication operator by $w$ induces an isometric isomorphism $\mathcal{L}_w^{a_i,1}(\dd v;X) \to L^{a_i,1}(\dd v;X)$, it follows that 
\[
    K\left(t,f;\mathcal{L}_w^{a_0,1}(\dd v;X),\mathcal{L}_w^{a_1,1}(\dd v;X)\right) = K\left(t,wf; L^{a_0,1}(\dd v,X), L^{a_1,1}(\dd v;X)\right),
\]
so that our task boils down to showing the interpolation property 
\begin{equation}
    \label{eq.interpolation.fact.last}
    \left(L^{a_0,1}(\dd v;X), L^{a_1,1}(\dd v;X)\right)_{\theta,r}=L^{a_{\theta},r}(\dd v;X). 
\end{equation}
For any function $f=f(v,x)$, write $F(v)=\|f(v)\|_X$. We claim that there holds:
\begin{equation}
    \label{claim.interpo}
    K(t,f;L^{a_0,1}(\dd v;X),L^{a_1,1}(\dd v;X)) = K(t,F;L^{a_0,1},L^{a_1,1}),
\end{equation}
where the spaces $L^{a_0,1}$ are the Lorentz spaces on $\mathbb{T}^d$ endowed with the Lebesgue measure. Once this claim is proved, then \eqref{eq.interpolation.fact.last} follows from \eqref{claim.interpo} and $(L^{a_0,1},L^{a_1,1})_{\theta,r}= L^{a,r}$, whose proof follows from the reiteration theorem and \eqref{eq.interpolation-lorentz1}. 

In order to prove \eqref{claim.interpo}, start with $f=f_0+f_1$, and write the associated $F$ as $F=G_1 + G_2$ with $G_i(v)=\frac{F(v)F_i(v)}{F_0(v)+F_1(v)}$ where $F_i(v)=\|f_i(v)\|_X$. The triangle inequality implies $F(v) \leq F_0(v)+F_1(v)$ and therefore $0\leq G_i(v) \leq F_i(v)$. Now, remark that $\{v : G_i(v)>t\} \subset \{v : F_i(v)>t\}$ so that 
\[
    \|G_i\|_{L^{a_i,1}} \leq \|F_i\|_{L^{a_i,1}} = \|f_i\|_{L^{a_i,1}(\dd v;X)},
\] 
which is enough to imply
\[
    K(t,F,L^{a_0,1},L^{a_1,1}) \leq K(t,f,L^{a_0,1}(\dd v;X),L^{a_1,1}(\dd v;X)).
\]
It remains to prove the reverse inequality. To this end, consider a decomposition $F=F_0+F_1$ and define $f_i(v)=\frac{F_i(v)}{F(v)}f(v)$ where we recall that $F(v)=\|f(v)\|_{X}$. Then $f=f_0+f_1$ and $\|f_i\|_{L^{a_i,1}(\dd v;X)}=\|F_i\|_{L^{a_i,1}(\dd v)}$ which is enough to conclude.
\end{proof}

\end{document}